\documentclass[11pt,reqno]{amsart}
\usepackage{amsmath,amsfonts,amsthm,amscd,amssymb,graphicx}
\usepackage{epstopdf}
\numberwithin{equation}{section}
\usepackage{fullpage}
\usepackage{color}

\numberwithin{equation}{section}

\newtheorem{theorem}{Theorem}[section]

\newtheorem{lemma}{Lemma}[section]

\newtheorem{proposition}{Proposition}[section]

\newtheorem{corollary}{Corollary}[section]

\theoremstyle{definition}
\newtheorem{definition}{Definition}[section]
\newtheorem{remark}{Remark}[section]

\newcommand{\RR}{{\mathbb R}}
\newcommand{\ds}{\displaystyle}

\newcommand\bp{\begin{pmatrix}}
\newcommand\ep{\end{pmatrix}}
\newcommand\be{\begin{equation}}
\newcommand\ee{\end{equation}}

\renewcommand\a{\alpha}
\renewcommand\b{\beta}
\renewcommand\o{\omega}
\newcommand\s{\sigma}

\newcommand{\sgn}{\operatorname{sgn}}

\title{Traveling waves for monostable reaction-diffusion-convection equations\\ with discontinuous density-dependent coefficients}

\author{Pavel Dr\'{a}bek}
\address{Department of Mathematics and NTIS, Faculty of Applied Sciences, University of West Bohemia, Univerzitní 8, Plze\'{n} 30100, Czech Republic}
\email{pdrabek@kma.zcu.cz}

\author{Soyeun Jung}
\address{Division of International Studies, Kongju National University, Gongju-si, Chungcheongnamdo, South Korea}
\email{soyjung@kongju.ac.kr}

\author{Eunkyung Ko}
\address{Major in Mathematics, College of Natural Sciences, Keimyung University, Daegu  42601, South Korea}
\email{ekko@kmu.ac.kr}

\author{Michaela Zahradn\'{i}kov\'{a}}
\address{Department of Mathematics and NTIS, Faculty of Applied Sciences, University of West Bohemia, Univerzitní 8, Plze\'{n} 30100, Czech Republic}
\email{mzahrad@kma.zcu.cz}

\begin{document}

\begin{abstract}
	This paper concerns wave propagation in a class of scalar reaction-diffusion-convection equations with $p$-Laplacian-type diffusion and monostable reaction.
	We introduce a new concept of a non-smooth traveling wave profile, which allows us to treat discontinuous diffusion with possible degenerations and singularities at 0 and 1, as well as only piecewise continuous convective velocity. Our approach is based on comparison arguments for an equivalent non-Lipschitz first-order ODE.
	We formulate sufficient conditions for the existence and non-existence of these generalized solutions and discuss how the convective velocity affects the minimal wave speed compared to the problem without convection. We also provide brief asymptotic analysis of the profiles, for which we need to assume power-type behavior of the diffusion and reaction terms. 	
\end{abstract}


\maketitle


\smallskip

{\bf  Key words}:  monostable reaction-diffusion-convection equation, degenerate and singular density-dependent diffusion,  discontinuous diffusion, discontinuous convection, non-Lipschitz reaction, traveling wave

\smallskip

{\bf AMS Subject Classifications}: 35C07, 35K57, 34B16

\smallskip

\section{Introduction}

We are concerned with traveling wave solutions to the reaction-diffusion-convection equation
\begin{equation}
	\label{rdc-eq}
	u_t+h(u)u_x=\left[d(u)|u_x|^{p-2}u_x\right]_x+g(u), \quad (x,t) \in \mathbb{R} \times [0,+\infty).
\end{equation}
Here $p>1$, $h=h(u)$ represents a nonlinear convective velocity, $g=g(u)$ is a monostable reaction term, i.e., $g(0)=g(1)=0$, $g>0$ in $(0,1)$, and $d=d(u)$ can be viewed as a nonlinear diffusion coefficient.
Our aim is to investigate the existence and properties of traveling wave solutions by employing very weak regularity assumptions on these functions. In particular, we consider $g \in C[0,1]$, while the diffusion and convection terms, represented by functions $d$ and $h$ respectively, may have jumps at finitely many points of $[0,1]$. Furthermore, the diffusion coefficient is not necessarily bounded near 0 and/or 1.
The precise hypotheses concerning $d$ and $h$ will be specified in the next section.

Particular cases of equation \eqref{rdc-eq} have been subject to extensive studies due to their wide applicability in modeling of biological, chemical and physical phenomena.
For $p=2$, numerous examples can be found in \cite{GK},
typically including one of the following three types of reaction terms:
Fisher-KPP-type (monostable), ignition- or combustion-type, and Nagumo-type (bistable), represented in Figure \ref{f:reaction}.
\begin{figure}[h]
	\includegraphics[width=14cm]{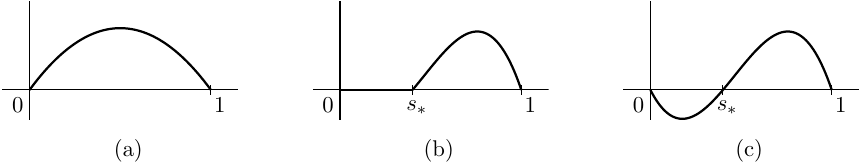}
	\caption{Sign conditions characterizing (a) Fisher-KPP-type (monostable); (b) combustion-type; (c) Nagumo-type (bistable) reaction terms}
	\label{f:reaction}
\end{figure}
This notion refers to their occurrence in classical applications from population genetics, combustion theory, and electrical pulse transmission, which can all be modeled by the semilinear reaction-diffusion equation:
\begin{equation*}
\label{FK}
u_t=u_{xx}+g(u),
\end{equation*}
see the works of Fisher \cite{Fisher}, Nagumo \cite{Nagumo}, and \cite{AW,GK} for other references.
Density-dependent dispersal, observed in many biological populations, can be accounted for by incorporating density-dependent diffusion coefficient $d=d(u)$:
\begin{equation*}
\label{FKd}
u_t=(d(u)u_x)_x+g(u),
\end{equation*}
cf. \cite{Murray, SG-Maini}.
The authors  of \cite{SG-Maini} explain how modeling of such phenomena leads to the above equation with a simply degenerate diffusion coefficient, i.e., $d(0)=0$, $d>0$ in $(0,1]$.
`Fast' diffusion equations with $d(u)=u^{-n}$, $0<n<2$, have been suggested to study situations in which low concentrations disperse very rapidly, see \cite{KingMcCabe} and the references therein.
Other notable examples include the porous media equation $u_t=(u^m)_{xx}$, $m>1$, with applications in gas flow, heat transfer and groundwater flow.
Equations modeling the transport of a liquid also often comprise convective effects related, for example, to gravitational pull.
The case of $p>1$ has also gained interest recently in the context of reaction-diffusion(-convection) equations, e.g., see \cite{AV,Enguica,Hamydy}.

Let us recall that a \emph{traveling wave solution} is a solution of the form $u(x,t)=U(x-ct)$,
where $U$ is the \emph{profile} of the wave, moving at speed $c \in \mathbb{R}$ and connecting the steady states $0$ and $1$.
Traveling waves in reaction-diffusion equation with and without convection have been widely studied in literature.
Existence results, methods and properties of traveling waves depend on particular considerations, such as the type of reaction, regularity assumptions, properties of the diffusion coefficient $d$, convective effects, the value of $p$, and their combinations.

Equation \eqref{rdc-eq} with $p=2$ was investigated in \cite{MM02} for continuous $d$, $g$ and $h$, with $d$ strictly positive.
In \cite{MM05}, the same problem is studied in the case of simply or doubly degenerate $d \in C^1[0,1]$, i.e., assuming $d(0)=0$ or $d(0)=d(1) =0$, respectively. Apart from general existence results, a detailed classification of solutions is provided, based on whether the equilibria 0 and/or 1 are attained.
Since the transition may not be smooth in general, the authors consider traveling wave solutions on a maximal interval $(\alpha,\beta)$ such that $U\in C(cl(\alpha,\beta)) \cap C^2(\alpha,\beta)$, $-\infty \leq \alpha<\beta \leq +\infty$.

In monostable equations,  there commonly exist infinitely many traveling wave solutions with different wave speeds $c \in [c^*,+\infty)$, e.g.,  see  \cite{AW, DrTa, DZ22, Enguica, MM03, Murray}. The minimal wave speed $c^*$ is positive in the absence of convection, with available estimates depending on $d$ and $g$. More precisely, Dr\'{a}bek and Zahradn\'{i}kov\'{a} proved  in a very general setting (\cite{DZ22}) that there exists $c^* \in (0, (p')^{1/{p'}}p^{1/p}\mu ^{1/p'}]$, where $p'$ is a conjugate number of $p$ and
\be
0<\mu:=\sup_{t \in (0,1)} \frac{(d(t))^{\frac{1}{p-1}}g(t)}{t^{p'-1}},
\ee
such that for each $c\geq c^*$ the equation \eqref{rdc-eq} with $h \equiv 0$ possesses a unique nonincreasing traveling wave satisfying boundary conditions $U(-\infty)=1$ and $U(+\infty)=0$.  The authors considered a large class of functions $d$ with diverse characteristics, allowing not only for degenerations and singularities at equilibria, but also jump discontinuities inside the interval $(0,1)$. 
The last assumption might be relevant in the modeling of phenomena involving sudden changes of the diffusion coefficient. Examples include problems from polymer dynamics, in which the diffusivity drops abruptly by several orders of magnitude beyond the gelation critical density, and processes related to hydrogen storage as a source of energy, cf. \cite{Strier} and the references therein. In contrast to \cite{Strier}, where a particular case of discontinuous diffusion was studied, \cite{DZ22} introduced a new concept of continuous solutions that allows for the treatment of very general diffusion coefficients, encompassing those previously considered in literature.

In the present work we employ a similar approach to study how convection $h$ affects the existence and properties of traveling wave solutions, in particular, the range of minimal wave speed $c^*$ compared to the problem without convection. Our existence result shows that the additional transport term $h$ causes  a `shift' of the threshold speed $c^*$
\be\label{c*_U_into}
c^* \in [ \max \{h(0), H(1)\},  ~ (p')^{\frac{1}{p'}} p^{\frac{1}{p} }\mu^{\frac{1}{p'}}+h_M],
\ee
where  $H(U):= \int_0^U h(s) ds$ and $h_M= \sup_{0 \leq t \leq 1} h(t)$. This generalizes the existence result in \cite{MM02, MM05} for the case $p=2$ with continuous $d$, $g$ and $h$. Since our methods do not require continuity of the convective velocity $h$ on $[0,1]$, we make explicit assumptions about discontinuities of both $d$ and $h$.

Following the strategy of \cite{DZ22} we establish the existence and nonexistence result by analyzing the transformed first order boundary value problem
\begin{equation}\label{ode_intro}
\begin{cases}
y'(t) = p' \left[ (c-h(t)) (y^+ (t))^{\frac{1}{p}} -f(t)  \right],  \quad t \in (0,1), \\
 y(0)=0=y(1),
\end{cases}
\end{equation}
where $y^+(t):=\max \{y(t),0\}$. In the case of $h \equiv 0$ (\cite{DZ22}),  a necessary condition $c>0$ for the existence of solution, which guarantees the uniqueness of solutions of backward initial value problem, was a key role in  performing comparison arguments. However, in our case, due to a sign change of $c-h(t)$, the previous idea is not applicable directly.  The novelty of our manuscript is overcoming the difficulty by first proving the positivity of solutions of the backward IVP. In fact, the positivity of $y(t)$ allows us to rewrite the ODE of \eqref{ode_intro} in terms of $z(t)=(y(t))^{1/p'},$ from which we obtain the uniqueness and comparison arguments.

The text is organized as follows.
In Section \ref{MainResults}, we introduce the definition of traveling wave profile and present our main existence and nonexistence results (Theorems \ref{Existence_U} and \ref{Nonexistence_U}).
In Section \ref{Equivalent first order ode}, we show that each profile is necessarily monotone on $\mathbb{R}$ and strictly monotone on an open interval $(z_0,z_1)$, bounded or unbounded. This allows to
reduce the second-order problem for the unknown profile $U$ into an equivalent first-order one \eqref{ode_intro} via a suitable substitution.
Section \ref{E_N_y} is devoted to the investigation of this first order boundary value problem \eqref{ode_intro}, resulting in Theorems \ref{exthm} and \ref{Nonexistence}.
Finally, in Section \ref{asymptotic analysis} we examine the asymptotic behavior of solutions. Due to our general assumptions, we need to assume power-type behavior of $d$ and $g$ near equilibria to classify solutions in terms of attaining values 0 and/or 1.

As our future work the study of traveling wave solutions for bistable or combustion reaction $g$ in our very general setting is interesting direction to carry out. The situation is very different for bistable or combustion nonlinearity, where convection might lead to the disappearance of the unique traveling wave solution that exists when $h \equiv 0$ (see \cite{MMM} for an overview in the case $p=2$).

\section{Main results} \label{MainResults}

As a starting point, we specify the hypotheses on the reaction term $g$, the diffusion coefficient $d$ and the convective velocity $h$, respectively.

\smallskip
\smallskip

\noindent \textbf{(H1)} $g:[0,1] \rightarrow \RR$ is continuous, but not necessarily smooth. In particular, $g$ satisfies
\be \label{r}
g(0)=g(1)=0, \quad \text{$g>0$ in $(0,1)$}.
\ee
\noindent \textbf{(H2)} $d:[0,1] \rightarrow \RR$ is nonnegative lower semicontinuous and $d>0$ in $(0,1)$. There exist $0=m_0<m_1< \cdots < m_r < m_{r+1}=1$ such that $d$ has jump discontinuity at $m_i$, $i=1, 2, \cdots, r$ and
\be
d|_{(m_i, m_{i+1})} \in C(m_i, m_{i+1}), \quad  i=0, 1, \cdots, r.
\notag
\ee

\noindent \textbf{(H3)} $h:[0,1] \rightarrow \RR$ is continuous at $0$ and bounded piecewise continuous on $[0,1]$. There exist $0=k_0<k_1<k_2< \cdots <k_q <k_{q+1}=1$ such that $h$ has jump discontinuity at $k_j$, $j=1, 2, \cdots, q$ and
\be
h|_{(k_j, k_{j+1})} \in C(k_j, k_{j+1}), \quad j=0,1, \cdots, q.
\notag
\ee
Since $h$ is bounded in $[0,1]$, throughout the paper we denote the supremum and infimum of $h$ in $[0,1]$ by
\be \label{maxminh}
h_M:=\ds \sup_{0 \leq t \leq 1} h(t) \quad \text{and} \quad h_m:=\ds \inf_{0 \leq t \leq 1} h(t).
\ee

\smallskip

When looking for traveling wave solutions of \eqref{rdc-eq}, we formally substitute $u(x,t)=U(z)$ with $z=x-ct$ into \eqref{rdc-eq}, which then reduces to an ordinary differential equation
\begin{equation}\label{eqU}
 (d(U)|U'|^{p-2} U')' +(c-h(U) )U' +g(U) =0,
 \end{equation}
where primes denote differentiation with respect to the wave coordinate $z$.
The main purpose of this research is to determine the region in $c$ for which the equation \eqref{eqU} has a solution $U=U(z)$ satisfying the boundary conditions
\be \label{bc}
\lim_{z \rightarrow -\infty} U(z) =1 ~~\mbox{and}~~\lim_{z\rightarrow \infty} U(z) = 0.
\ee

\subsection{Definition of solution}
\label{DefSol}

Since $d(U)$ and $h(U)$ have discontinuities on $[0,1]$, we cannot expect a classical solution $U=U(z)$ of \eqref{eqU}, and hence we need the definition of more general solutions of \eqref{eqU}. For this definition we denote
\be \label{MKN}
\begin{split}
M_U: & = \{z \in \mathbb{R} : U(z) = m_i, ~i=1,2, \dots r \}, \\
K_U: & = \{z \in \mathbb{R} : U(z) = k_j, ~j=1,2, \dots q \}, \\
\text{and} \quad N_U: & = \{z \in \mathbb{R}: U(z)=0 ~\mbox{or}~ U(z)=1  \}.
\end{split}
\ee
Moreover, we say that a function $U$ is piecewise $C^1$ on an interval $I$ if $U \in C(I)$ and there is a discrete set $D \subset I$ such that $U \in C^1(I \setminus D)$. Following \cite{DZ22}, we define what we mean by a solution of \eqref{eqU}.

\smallskip

\begin{definition} [cf. \cite{DZ22}] \label{sdef}
A function $U :\mathbb{R} \rightarrow [0,1] \in C(\RR)$ is called a solution of \eqref{eqU} if it satisfies the following properties:

\smallskip

\noindent (a) $U \in C^1(\RR \setminus D_U)$, where $D_U:=\partial M_U \cup \partial N_U$ is  a discrete set in $\mathbb{R}$.

\smallskip

\noindent (b) For any $z \in \partial M_U$ there exist finite one sided derivatives $U'(z-), U'(z+)$ and
\be \label{Ldef}
L(z) := |U'(z-)|^{p-2} U'(z-) \lim_{\xi \rightarrow z-} d(U(\xi)) = |U'(z+)|^{p-2} U'(z+)\lim_{\xi \rightarrow z+ } d(U(\xi)).
\ee

\noindent (c) A function $v: \mathbb{R} \rightarrow \mathbb{R}$ defined by
 \be\label{v}
 v(z) :=
 \left\{\begin{array}{lll}
 d(U)|U'(z)|^{p-2} U'(z), & z \not\in \partial M_U \cup \partial N_U,\\
 0, & z\in  \partial N_U,\\
 L(z), &z \in \partial M_U
 \end{array}\right.
 \ee
is continuous, and the following integral form holds for any $z , \hat{z} \in \mathbb{R}$
\begin{equation}\label{eqv}
 v(\hat{z}) - v(z) +c (U(\hat{z}) - U(z)) - (H (U(\hat{z})) - H(U(z))) + \int_z^{\hat{z}} g(U(\tau)) d\tau =0,
\end{equation}
where $H(U):= \ds\int_0^U h(s) ds$. Moreover, $\ds \lim_{z \rightarrow \pm \infty} v(z)=0$.
 \end{definition}

\begin{remark} We notice that $M_U$ has no interior points, that is, $\partial M_U=M_U$. If it has an interior point, there exist $i \in \{1, 2, \cdots, r \}$ and an open interval $I$ such that $U(z)=m_i$ for all $z \in I$. Then $U'(z)=0$ for all $z \in I$, and hence it follows from \eqref{eqv} that $\int_z^{\hat{z}} g(U(\tau)) d\tau =0$ for all $z, \hat z \in I$, $z \neq \hat z$, in contradiction to the fact that $g(s)>0$ for all $s \in (0,1)$. In a similar manner we have that $\partial K_U=K_U$.
\end{remark}

\begin{remark}We remark that $U'(z)$ exists for all $z \in K_U \setminus M_U$.  Indeed, if $U'(z+) \not= U'(z-)$ at $z \in K_U\setminus M_U$, since $d(U(z))$ is continuous at $z \in K_U\setminus M_U$, we see that
$$d(U(z-))|U'(z-)|^{p-2} U'(z-) \not= d(U(z+))|U'(z+)|^{p-2} U'(z+).$$
This contradicts to the continuity of $v(z)$, which is  required in Definition \ref{sdef}(c) because $H(U(z))$ is  continuous at $z \in K_U\setminus M_U$ even if $h(z)$ is discontinuous at $z \in K_U\setminus M_U.$
  However, the derivative of $v=v(z)$ may not exist at $z \in K_U \setminus M_U$ even if $U'(z)$ exists. See the remark below.
\end{remark}

\begin{remark}\label{remn_h}
\noindent $(i)$ Let $z \not\in M_U \cup K_U \cup N_U$ and $\hat z=z+\tau$, $\tau \neq 0$. Assume $|\tau|$ is sufficiently small that $\hat z \notin M_U\cup K_U \cup N_U$. Dividing \eqref{eqv} by $\tau$ and making $\tau \rightarrow 0$ we obtain
\be
\lim_{\tau \rightarrow 0} \frac{v(z+\tau)-v(z)}{\tau}+(c -h(U(z)))U'(z) +g(U(z))=0
\notag
\ee
because the derivatives $U'(z)$ and $\frac{d}{dz}H(U(z))$ exist. Thus $v'(z)$ exists for all $z \not\in M_U\cup K_U\cup N_U$ and it holds
\be
v'(z)+(c -h(U(z)))U'(z) +g(U(z))=0
\notag
\ee
which implies that $v \in C^1(\RR \setminus (M_U \cup K_U \cup N_U))$.

\smallskip

\noindent $(ii)$ Let $z \in M_U \cup K_U$ and $\hat z=z+\tau$. Dividing \eqref{eqv} by $\tau$ and making $\tau \rightarrow 0-$ we obtain
\be
v'(z-)+\Big(c -\lim_{\xi \rightarrow z-}h(U(\xi))\Big)U'(z-) +g(U(z))=0,
\notag
\ee
from which the left-hand derivative $v'(z-)$ exists as a finite number for all $z \in M_U \cup K_U$.
Similarly, making $\tau \rightarrow 0+$, the right-hand derivative $v'(z+)$ exists for all $z \in M_U \cup K_U$, and it satisfies
\be
v'(z+) +\Big(c -\lim_{\xi \rightarrow z+}h(U(\xi))\Big)U'(z+) +g(U(z))=0.
\notag
\ee
In particular, if $z \in K_U \setminus M_U$, then $U'(z-)=U'(z+)$ and
\be
v'(z\pm)+\Big(c -\lim_{\xi \rightarrow z\pm}h(U(\xi))\Big)U'(z) +g(U(z))=0.
\notag
\ee
If $z \in M_U \setminus K_U$, then $\ds \lim_{\xi \rightarrow z-}h(U(\xi))=\lim_{\xi \rightarrow z+}h(U(\xi))=h(U(z))$ and
\be
v'(z\pm)+\Big(c -h(U(z))\Big)U'(z\pm) +g(U(z))=0.
\notag
\ee
\end{remark}

\subsection{Existence and nonexistence of traveling waves}

The results of Section \ref{Equivalent first order ode} and Section \ref{E_N_y} lead to the existence and nonexistence of a traveling wave solution for the second order boundary value problem \eqref{eqU}, \eqref{bc}, namely
\begin{equation*}\label{odeU}
	\begin{cases}
		(d(U)|U'|^{p-2} U')' +(c-h(U) )U' +g(U) =0, \quad z \in \RR,  \\
		\ds \lim_{z \rightarrow -\infty} U(z) =1 ~~\mbox{and}~~\ds \lim_{z\rightarrow +\infty} U(z) = 0.
	\end{cases}
\end{equation*}

\begin{theorem} [Existence]\label{Existence_U} Assume that $g$, $d$ and $h$ satisfy \textbf{(H1)}--\textbf{(H3)}, and let
\begin{equation} \label{mu_U}
	0< \mu := \sup_{t\in (0,1)} \frac{(d(t))^{\frac{1}{p-1}}g(t)}{t^{p'-1}} < + \infty,
\end{equation}
where $\frac{1}{p}+\frac{1}{p'}=1$.  Then, in the case $h(0)>H(1)$, there exists a number
\be\label{c*_U}
c^* \in [ h(0), ~ h_M +(p')^{\frac{1}{p'}} p^{\frac{1}{p} }\mu^{\frac{1}{p'}}]
\ee
such that the problem \eqref{eqU}, \eqref{bc} has a unique traveling wave $U=U(z)$, $z \in \RR$, if and only if $c\geq c^*$. In the case $h(0) \leq H(1)$, the same result holds for
\be \label{H(1)_U}
c^* \in (H(1), ~h_M+(p')^{\frac{1}{p'}} p^{\frac{1}{p} }\mu^{\frac{1}{p'}}  ].
\ee
Moreover, recalling the notation from \eqref{MKN}, $U$ has the following properties: \smallskip
\begin{itemize}
\item[(i)] $U$ is strictly decreasing in the open interval $(z_0, z_1)=\{z \in \RR: 0<U(z)<1\}$ where $-\infty \leq z_0 < 0 <z_1 \leq +\infty$, $U(0)=\frac{1}{2}$, $U(z)=1$ for $z \in (-\infty, z_0]$, and $U(z)=0$ for $z \in [z_1, +\infty)$.

\smallskip

\item[(ii)] For $i=1, \cdots, r$, let $\zeta_i \in (z_0, z_1)$ be such that $U(\zeta_i)=m_i$, $\zeta_0 = z_1$, $\zeta_{r+1} = z_0$. Then $U \in C(\RR)$ is a piecewise $C^1(\RR)$--function in the sense that
\be
U|_{(\zeta_{i+1}, \zeta_{i})} \in C^1(\zeta_{i+1}, \zeta_{i}), \quad i=0, 1, \cdots, r
\ee
and one-sided derivatives $\ds U'(\zeta_i \pm)=\lim_{z \rightarrow \zeta_i \pm}U'(z)$, $i=1, 2, \cdots, r$, exist and are finite.

\smallskip

\item[(iii)] $U$ satisfies $\ds \lim_{z \rightarrow z_0+} d(U(z))|U'(z)|^{p-2}U'(z)=\lim_{z \rightarrow z_1-} d(U(z))|U'(z)|^{p-2}U'(z)=0$
and the transition condition
\[
|U'(\zeta_i-)|^{p-2}U'(\zeta_i-) \lim_{z \to \zeta_i-} d(U(z))=|U'(\zeta_i+)|^{p-2}U'(\zeta_i+) \lim_{z \to \zeta_i+} d(U(z)),  \quad  i=1, 2, \cdots, r.
\]

\smallskip

\item[(iv)] $U$ satisfies the problem \eqref{eqU}, \eqref{bc} (in a classical sense) on $\RR \setminus (M_U \cup K_U \cup \partial N_U)$.
\end{itemize}
\end{theorem}

\begin{theorem} [Nonexistence]\label{Nonexistence_U} Assume $g$, $d$ and $h$ satisfy \textbf{(H1)}--\textbf{(H3)}, and let
\be \label{nu_U}
0< \ds \nu:=\liminf_{t \rightarrow 0+}\frac{(d(t))^{\frac{1}{p-1}}g(t)}{t^{p'-1}}.
\ee
If
\be \label{condition_nonexist_U}
h(0) < c< h(0)+(p')^{\frac{1}{p'}}p^{\frac{1}{p}}\nu^{\frac{1}{p'}},
\ee
then the problem \eqref{eqU}, \eqref{bc} has no traveling wave solution $U=U(z)$, $z\in \RR$. In particular, if $\nu =+\infty$, the problem \eqref{eqU}, \eqref{bc} has no traveling wave solution for any $c > h(0)$.
\end{theorem}

\section{Equivalent first order ode} \label{Equivalent first order ode}

This section is devoted to providing a first order ODE equivalent to the profile equation \eqref{eqU} with the boundary conditions \eqref{bc}. As discussed in the introduction, the existence and nonexistence of solution of \eqref{eqU} and \eqref{bc} will be analyzed via the transformed first order boundary value problem \eqref{ode_intro}.

\subsection{Monotonicity of a traveling wave profile}

Before converting the ODE \eqref{eqU} and \eqref{bc} to \eqref{ode_intro} we first observe basic properties of a profile $U=U(z)$ of \eqref{eqU} and \eqref{bc}. Passing to the limit for $z\rightarrow -\infty$ in \eqref{eqv} and replacing $\hat{z}$ by $z$, we have
\begin{equation}\label{left}
v(z) + c (U(z) -1) - (H (U(z)) - H(1)) +\int_{-\infty}^{z} g(U( \tau) ) d\tau =0
\end{equation}
for all $z \in \mathbb{R}.$
Also, passing to the limit for $\hat{z}\rightarrow +\infty$ in \eqref{eqv} gives
\begin{equation}\label{right}
-v(z) - c U(z)  + H(U(z))  +\int_{z}^{+\infty} g(U( \tau) ) d\tau =0
\end{equation}
for all $z \in \mathbb{R}.$ Passing to the limit for $\hat{z} \rightarrow \infty$ and $z \rightarrow - \infty$ in \eqref{eqv}, it follows that
\begin{equation}\label{ne1}
c- H(1) = \int_{-\infty}^{+\infty} g(U( \tau) ) d\tau.
\end{equation}
Since $g > 0$ in $(0,1)$ and $U$ satisfies \eqref{bc},  we can see that
\begin{equation}\label{nc1}
c >H(1)
\end{equation}
which is a necessary condition for the existence of solutions of \eqref{eqU} and \eqref{bc}. Another necessary condition will be stated in Lemma \ref{lemnc2}.

\bigskip

We now investigate the monotonicity of a solution of \eqref{eqU} with \eqref{bc}. The proofs of the following two lemmas are essentially the same as in the case $h \equiv 0$ (\cite[Lemma 3.1 and Lemma 3.2]{DZ22}). We provide brief proofs for reader's convenience.

\begin{lemma}\label{lem01_h}
Let $U(z)$, $z \in \mathbb{R}$, be a solution of \eqref{eqU} and \eqref{bc}, and  assume $\xi \in N_U$. Then the following two alternatives hold:

\smallskip

\noindent ~$(i)$  if $U(\xi) =0$ then $U(z)=0$ for all $z \geq \xi$;

\smallskip

\noindent $(ii)$ if $U(\xi) =1$ then $U(z)=1$ for all $z \leq \xi$.
\end{lemma}
\begin{proof} We only prove (i). The proof of (ii) goes similarly. Let $ U(\xi)=0$ and assume by contraction that there exists $\xi^*>\xi$ such that $0<U(\xi^*)<1$. Then $g(U(\xi^*)) >0$ and hence $\int_{\xi}^{+\infty} g(U( \tau) ) d\tau>0$. Recalling Definition \ref{sdef}(c), we get $v(\xi)=0$ and $H(U(\xi))=0$. From \eqref{right} with $z=\xi$, we conclude that $\int_\xi^\infty g(U(\tau)) d\tau =0$, a contradiction.
\end{proof}

\begin{lemma}\label{lemmon_h}
Let $U(z)$, $z\in  \mathbb{R}$, be a solution of \eqref{eqU} and \eqref{bc}. Then $U$ is nonincreasing in $ \mathbb{R}.$
In particular, for $z \not\in N_U$, we have $U'(z) <0$ if $z \not\in M_U$, and $U'(z-) <0$ and $U'(z+) <0$ if $z \in M_U.$
That is,  $U$ is strictly decreasing for each $z \in \mathbb{R}$ at which $0<U(z) <1.$
 \end{lemma}
 \begin{proof} We first claim that once there is $\xi \not\in N_U$ such that $U'(\xi-)=0$ or $U'(\xi+)=0$, $\xi$ becomes a point of strict local maximum of $U$. In fact, if $\xi \not\in N_U$ satisfies either $U'(\xi-)=0$ or $U'(\xi+)=0$, then it satisfies $U'(\xi-)=0=U'(\xi+)$. Otherwise, it contradicts to \eqref{Ldef} because $\ds \lim_{z \rightarrow \xi\pm}d(U(z)) > 0$ for  $\xi \not\in N_U$. Let $\xi \not\in N_U$ such that $U'(\xi-)=0$. It follows from Remark \ref{remn_h} that
$$v'(\xi-) =-g(U(\xi)) < 0$$
regardless of whether $\xi \in M_U \cup K_U$ or $\xi \not\in M_U \cup K_U$.
Since $v(\xi) =0$ and $v'(\xi-)<0$, there exists a left neighborhood $\mathcal{U}_{-}(\xi)$ such that $\mathcal{U}_{-}(\xi) \cap N_U = \emptyset$ and
\be \label{pv}
v(z)>0 \quad \text{for all $z \in \mathcal{U}_{-}(\xi)$.}
\ee
Since $d(U(z)) >0$ in $z \in \mathcal{U}_{-}(\xi)$, it holds from Definition \ref{sdef} and \eqref{pv} that $U'(z-) >0$ and $U'(z+) >0$ in $z \in \mathcal{U}_{-}(\xi)$. Thus we deduce that $U(z) < U(\xi)$ for all  $z \in  \mathcal{U}_-(\xi)$. In a similar manner, since $U'(\xi+) =0$, there exists a right neighborhood $\mathcal{U}_+(\xi)$ such that
 $U(z)  < U(\xi)$ for all  $z \in  \mathcal{U}_+(\xi)$. Therefore, $\xi$ is the point of strict local maximum of $U$.

Next, we prove that $U(z)$ cannot have a strict local maximum on $\RR \setminus N_U$. Let $\xi \notin N_U$  and $U(\xi)$ be a strict local maximum. Since $U(\xi)<1$ and $U(z) \rightarrow 1$ as $z \rightarrow -\infty$, there exists $\xi_* \in (- \infty, \xi)$ such that $U(\xi) \leq U(\xi_*)<1$. Since $U \in C(\RR)$, let $\xi^* \in [\xi_*, \xi]$ be the global minimum of $U$ over $[\xi_*, \xi]$. Then it follows that $U(\xi^*) < U(\xi) \leq U(\xi_*) <1$, and so $\xi^* \in (\xi_*, \xi)$. Moreover, $\xi^* \notin N_U$, i.e., $0<U(\xi^*)$. Indeed, if $U(\xi^*)=0$, Lemma \ref{lem01_h}(i) gives $U(\xi)=0$, a contradiction. Now, we can derive contradictions in either case $\xi^* \in M_U$ or $\xi^* \not\in M_U$ regardless of whether $\xi^*$ is in $K_U$. If $\xi^* \not\in M_U$, then by Remark \ref{remn_h}, $U'(\xi^*)$ exists and $U'(\xi^*)=0$ as $\xi^*$ is a local minimizer of $U$, in contradiction to the above claim. If $\xi^* \in M_U$, we see from Definition \ref{sdef}(b) that  $\sgn U'(\xi^*-) = \sgn U'(\xi^*+)$. However, since $U(\xi^*)$ is a local minimum,  we conclude that $U'(\xi^*-)= U'(\xi^*+)=0$ which also contradicts the claim. It completes the proof.
\end{proof}

\subsection{Transformation into the first order boundary value problem}

We now reduce the second order boundary value problem \eqref{eqU} and \eqref{bc} to a first order boundary value problem. We apply the reduction established in \cite{DZ22} under the hypotheses \textbf{(H1)}--\textbf{(H3)}. As discussed in Lemma \ref{lem01_h} and Lemma \ref{lemmon_h}, a solution $U=U(z)$ of \eqref{eqU} and \eqref{bc} is nonincreasing in $\RR$. More precisely, $U$ is strictly decreasing on a suitable open interval $(z_0, z_1) \subseteq \RR$ with $-\infty \leq z_0 <z_1 \leq +\infty$ and
\be \label{asy}
\begin{split}
& \lim_{z \rightarrow z_0+} U(z)=1 \quad \text{and} \quad  U(z)=1  ~  \text{if} ~  -\infty < z \leq z_0, \\
& \lim_{z \rightarrow z_1-} U(z)=0 \quad  \text{and} \quad  U(z)=0  ~ \text{if} ~  z_1 \leq z <+\infty.
\end{split}
\ee
In what follows, recalling \eqref{MKN},  let
\be \label{MK_U}
M_U=\{\zeta_1, \zeta_2, \cdots, \zeta_r \} \quad \text{and} \quad K_U=\{\kappa_1, \kappa_2, \cdots, \kappa_q \},
\ee
where
\be
U(\zeta_i)=m_i, \quad i=1, 2, \cdots, r \quad  \text{and} \quad U(\kappa_j)=k_j, \quad  j=1, 2, \cdots, q.
\notag
\ee
We now rewrite
\be \label{MunionK}
M_U \cup K_U=\{\xi_1, \xi_2, \cdots, \xi_n\} \quad \text{and} \quad U(\xi_i)=s_i, \quad i=1, 2, \cdots, n.
\ee
That is, by setting $s_0=0$ and $s_{n+1}=1$ we have that
\be
0=s_0 < s_1< s_2 < \cdots < s_n < s_{n+1}= 1,
\notag
\ee
and $d$ or/and $h$ has jump discontinuity at $s_i$, $i=1, 2, \cdots, n$. Moreover,
\be
d|_{(s_i, s_{i+1})}, ~ h|_{(s_i, s_{i+1})} \in C(s_i, s_{i+1}), \quad  i=0, 1, \cdots, n.
\notag
\ee
Lastly, recalling \eqref{asy}, let us set $\zeta_0=\kappa_0=\xi_0=z_1$ and $\zeta_{r+1}=\kappa_{q+1}=\xi_{n+1}=z_0$.

\smallskip

It follows from Definition \ref{sdef}(a) that $U \in C(\RR)$ is a piecewise $C^1(\RR)$--function in the sense that 
\be
U|_{(\zeta_{i+1}, \zeta_{i})} \in C^1(\zeta_{i+1}, \zeta_{i}), \quad i=0, 1, \cdots, r.
\notag
\ee
Therefore, $U|_{(z_0, z_1)}$ has a continuous strictly decreasing inverse function $U^{-1}: (0,1) \rightarrow (z_0, z_1)$, $z=U^{-1}(U)$, such that  $U^{-1}|_{(m_i, m_{i+1})} \in C^1(m_i, m_{i+1})$, $i=0, 1, \cdots, r$ with finite limits
\be
\lim_{U \rightarrow m_i-}\frac{dz}{dU}=\Big(\lim_{z \rightarrow \zeta_i+}\frac{dU}{dz} \Big)^{-1} \quad  \text{and} \quad  \lim_{U \rightarrow m_i+}\frac{dz}{dU}=\Big(\lim_{z \rightarrow \zeta_i-}\frac{dU}{dz} \Big)^{-1}, \quad i=1, 2, \cdots, r.
\notag
\ee
Recalling the function $v=v(z)$ defined in \eqref{v}, we make the substitution
\be \label{wdef}
w(U):=v(z(U)), \quad U \in (0,1).
\ee
As discussed in Remark \ref{remn_h}, $v \in C(\RR)$ and $v  \in C^1(\RR \setminus (M_U \cup K_U \cup N_U))$. Here, we emphasize that $v=v(z)$ is not differentiable at $z \in K_U \setminus M_U$ even if $U'(z)$ exists.  Thus we immediately see that
\be
w \in C(0,1) \quad \text{and} \quad  w|_{(s_i, s_{i+1})} \in C^1(s_i, s_{i+1}),  \quad i=0, 1, \cdots, n,
\notag
\ee
with finite limits
$\ds \lim_{U \rightarrow s_i+}w'(U)$ and $\ds \lim_{U \rightarrow s_i-}w'(U)$, $i=1, 2, \cdots, n$. Moreover, direct computation gives that for any $z \in (\xi_{i+1}, \xi_{i})$ and $U \in (s_i, s_{i+1})$, $i=0, 1, \cdots, n$,
\be
\frac{d}{dz}v(z)=\frac{d}{dz}w(U(z))=\frac{dw}{dU}U'(z)=-\frac{dw}{dU}\Big|\frac{v(z)}{d(U(z))}  \Big|^{\frac{1}{p-1}}=-\frac{dw}{dU}\Big|\frac{w(U)}{d(U)} \Big|^{\frac{1}{p-1}}.
\notag
\ee
Here, we have used  $v(z)=-d(U(z))|U'(z)|^{p-1}$ because $U'(z)<0$ on $(\xi_{i+1}, \xi_{i})$. Consequently, from \eqref{eqU} the function $w=w(U)$ satisfies
\be
-\frac{dw}{dU}\Big|\frac{w(U)}{d(U)} \Big|^{\frac{1}{p-1}}-(c-h(U))\Big|\frac{w(U)}{d(U)} \Big|^{\frac{1}{p-1}}+g(U)=0, \quad U \in (s_i, s_{i+1}), \quad i=0, 1, \cdots, n.
\notag
\ee
Multiplying both sides by $d(U)^{1/(p-1)}$, writing $t$ instead of $U$, and making the substitution
\be \label{y}
y(t):=|w(t)|^{p'},  \quad p'=\frac{p}{p-1},
\ee
we arrive at
\be
\frac{1}{p'}\frac{dy}{dt}=(c-h(t))(y(t))^{\frac{1}{p}}-(d(t))^{\frac{1}{p-1}}g(t), \quad t \in (s_i, s_{i+1}), \quad i=0, 1, \cdots, n.
\notag
\ee
Finally, by setting
\be \label{f}
f(t):=(d(t))^{\frac{1}{p-1}}g(t)
\ee
we obtain the following first order ODE
\be \label{ode2}
y'(t) = p' \left[ (c-h(t)) (y (t))^{\frac{1}{p}} -f(t)  \right],  \quad t \in (0,1) \setminus \bigcup_{i=1}^n \{s_i \}.
\ee
Furthermore, we deduce from \eqref{bc}, \eqref{wdef} and Definition \ref{sdef}(c) that $\ds \lim_{t \rightarrow 0+}w(t)=\lim_{t \rightarrow 1-}w(t)=0$. Therefore, $y=y(t)$ satisfies the boundary conditions
\be \label{ode2_b}
y(0)=y(1)=0
\ee
and $y \in C[0,1]$.
\bigskip

On the other hand, let us derive \eqref{eqU} and \eqref{bc} from the first order boundary value problem \eqref{ode2} and \eqref{ode2_b}. Suppose that $y=y(t) \in C[0,1]$ is a positive solution of \eqref{ode2} and \eqref{ode2_b}. First, set
\be
w(t):=-(y(t))^{\frac{1}{p'}}, \quad t \in (0,1).
\notag
\ee
By writing $U$ instead of $t$ we immediately see that $w=w(U) \in C(0,1)$ satisfies
\be \label{w(U)}
-\frac{dw}{dU}\Big|\frac{w(U)}{d(U)} \Big|^{\frac{1}{p-1}}-(c-h(U))\Big|\frac{w(U)}{d(U)} \Big|^{\frac{1}{p-1}}+g(U)=0, \quad U \in (0,1) \setminus \bigcup_{i=1}^n \{s_i \}.
\ee
We now define
\be \label{z(U)}
z(U)=-\int_{\frac{1}{2}}^U \Big|\frac{d(s)}{w(s)}\Big|^{\frac{1}{p-1}}ds, \quad U \in (0,1).
\ee
Then $z=z(U)$ is continuous strictly decreasing in $(0,1)$ with $z(\frac{1}{2})=0$ and $z>0$ in $(0,\frac{1}{2})$ while $z<0$ in $(\frac{1}{2},1)$, so that $z=z(U)$ maps $(0,1)$ onto $(z_0, z_1)$ where $-\infty \leq z_0 < z_1 \leq +\infty$. Thus, it has a continuous strictly decreasing inverse function $z^{-1}:(z_0, z_1) \rightarrow (0,1)$, $U=z^{-1}(z)$, such that
\be
\lim_{z \rightarrow z_0+}U(z)=1 \quad \text{and} \quad \lim_{z \rightarrow z_1-}U(z)=0.
\notag
\ee
Also, it follows from \eqref{z(U)} that $U|_{(\zeta_{i+1}, \zeta_{i})} \in C^1(\zeta_{i+1}, \zeta_{i})$, $i=0, 1, \cdots, r$, with
\be \label{dUdz}
\frac{dU}{dz}=\frac{1}{\frac{dz}{dU}}=-\Big|\frac{w(U)}{d(U)}\Big|^{\frac{1}{p-1}}<0, \quad U \in (m_i, m_{i+1}),
\ee
from which we deduce that
\be \label{w(U(z))}
w(U(z))=-d(U(z))\Big| \frac{dU}{dz} \Big|^{p-1}, \quad z \in (\zeta_{i+1}, \zeta_{i}), \quad i=0, 1, \cdots, r.
\ee
Setting $v(z):=w(U(z))$, since $y^{\prime}(U)$ is undefined at $U(\xi_i)=s_i$, $i=1, 2, \cdots, n$, we see from \eqref{dUdz} that
\be \label{dvdz}
\frac{dv}{dz}=\frac{dw}{dU}\frac{dU}{dz}=-\frac{dw}{dU}\Big|\frac{w(U(z))}{d(U(z))}\Big|^{\frac{1}{p-1}}, \quad z \in (\xi_{i+1}, \xi_{i}), \quad i=0, 1, \cdots, n.
\ee
Applying \eqref{dUdz}--\eqref{dvdz} to \eqref{w(U)} yields
\be
\frac{d}{dz}\Big[d(U(z))\Big|\frac{dU}{dz}\Big|^{p-2}\frac{dU}{dz}\Big]+(c-h(U(z)))\frac{dU}{dz}+g(U(z))=0, \quad z \in (\xi_{i+1}, \xi_{i}), \quad i=0, 1, \cdots, n.
\notag
\ee
Moreover, recalling $\zeta_0=z_1$, $\zeta_{r+1}=z_0$, $w=w(U) \in C[0,1]$, and $U=U(z) \in C(z_0, z_1)$, we get from \eqref{ode2_b} and \eqref{w(U(z))} that
\be
\lim_{z \rightarrow z_0+} d(U(z))|U'(z)|^{p-2}U'(z)=\lim_{z \rightarrow z_1-} d(U(z))|U'(z)|^{p-2}U'(z)=0,
\notag
\ee
and the following one--sided limits are equal as a finite number
\be
\lim_{z \rightarrow \zeta_i-} d(U(z))|U'(z)|^{p-2}U'(z)=\lim_{z \rightarrow \zeta_i+} d(U(z))|U'(z)|^{p-2}U'(z), \quad i=1, 2, \cdots, r.
\notag
\ee
Since $U$ is strictly decreasing in $(z_0, z_1)$,
\be
\lim_{z \rightarrow \zeta_i\mp} d(U(z))=\lim_{m \rightarrow m_i\pm}d(m), \quad i=1, 2, \cdots, r.
\notag
\ee
By \eqref{dUdz} one-sided derivatives $U'(\zeta_i \pm)$ exist as
\be
U'(\zeta_i \pm)=-\Big|\frac{w(m_i)}{\ds \lim_{m \rightarrow m_i \pm} d(m)}\Big|^{\frac{1}{p-1}}, \quad i=1, 2, \cdots, r,
\notag
\ee
(notice that $\ds \lim_{m \rightarrow m_i \pm} d(m) \ne 0$ due to \textbf{(H2)}),
from which we deduce the transition condition
\be \label{transition}
|U'(\zeta_i-)|^{p-2}U'(\zeta_i-) \lim_{m \rightarrow m_i-} d(m)=|U'(\zeta_i+)|^{p-2}U'(\zeta_i+) \lim_{m \rightarrow m_i+} d(m), \quad i=1, 2, \cdots, r.
\ee

\smallskip

We now summarize the above observation in the following proposition.

\begin{proposition} \label{prop} A piecewise $C^1$--function $U:\RR \rightarrow [0,1]$ is a solution  of \eqref{eqU} and \eqref{bc} if and only if $y:[0,1] \rightarrow \RR$ is a continuous positive solution of \eqref{ode2} and \eqref{ode2_b}. In particular, a solution $U$ of \eqref{eqU} and \eqref{bc} is determined (up to translation) by a solution $y$ of \eqref{ode2} and \eqref{ode2_b} and vice versa.
\end{proposition}

Thanks to Proposition \ref{prop}, in order to study the existence and uniqueness for $U$ of \eqref{eqU} and \eqref{bc}, we focus on the reduced first order boundary value problem \eqref{ode2} and \eqref{ode2_b} in the next section.

\smallskip

We now notice that another necessary condition for $c$ can be obtained from \eqref{ode2} and \eqref{ode2_b}.

\begin{lemma} \label{lemnc2}
Let $y$ be a positive solution of \eqref{ode2} and \eqref{ode2_b}.
Then
\begin{equation}\label{nc2}
c\geq h(0).
\end{equation}
\end{lemma}
\begin{proof}
We assume by contradiction that 
$c<h(0)$. Since $h$ is continuous at $0$,
there exists $\delta >0$ such that $c <h(t)$ for all $t \in [0,\delta].$
Integrating \eqref{ode2} over $[0,\delta]$ gives
\begin{eqnarray*}
y(\delta) = \int_0^\delta y'(\tau) d\tau = p' \left[\int_0^\delta (c- h(\tau)) (y(\tau))^{\frac{1}{p}}d\tau -\int_0^\delta f(\tau) d\tau \right]<0
\end{eqnarray*}
which  contradicts to 
the positivity of $y$. Hence, $c \geq h(0)$.
\end{proof}

\begin{remark} Combining the above lemma with \eqref{nc1}, we consider
\be\label{ncs}
c \geq h(0) \quad \text{and} \quad c>H(1)
\ee
as necessary conditions for the existence of solutions of \eqref{ode2} and \eqref{ode2_b}. The necessary condition \eqref{nc1} can also be obtained from \eqref{ode2}. Indeed, by dividing  \eqref{ode2} by $y(t)^{\frac{1}{p}}$ and integrating it over $[0,1]$ we arrive at
\be \label{c>H(1)}
0=c-H(1)-\int_0^1 \frac{f(t)}{y(t)^{\frac{1}{p}}} dt,
\ee
from which we deduce that $c>H(1)$.

We here emphasize that $c$ can be negative. In the case that $h \equiv 0$ in \cite{DZ22},  a necessary condition $c >0$ was a crucial key to studying the boundary value problem \eqref{ode2} and \eqref{ode2_b}. Thus, in the present work, a sign change of $c-h(t)$ in the equation \eqref{ode2} causes a technical difficulty as compared to \cite{DZ22}. For this reason, we focus on how we overcome this difficulty in the next section.
\end{remark}

\section{Existence, uniqueness, and nonexistence of the first order ode} \label{E_N_y}

In this section we study the existence and uniqueness of solutions for the first order boundary value problem
\begin{equation}\label{ode}
\begin{cases}
y'(t) = p' \left[ (c-h(t)) (y^+ (t))^{\frac{1}{p}} -f(t)  \right],  \quad t \in (0,1), \\
 y(0)=0=y(1),
\end{cases}
\end{equation}
where $y^+(t):=\max \{y(t),0\}$, $p>1$ and $p'=\frac{p}{p-1}$ is the exponent conjugate. We will consider this problem independently of the original problem \eqref{eqU} and \eqref{bc}, and impose appropriate hypotheses on $f$ and $h$ at each step. As $U$ is not a solution of \eqref{eqU} in a classical sense, we should employ the concept of a solution $y=y(t)$ for \eqref{ode} in the sense of Carath\'{e}odory (\cite[Chapter 3]{W}). That is, a solution $y(t)$ is absolutely continuous and satisfies the first order ODE almost everywhere in $(0,1)$ and the boundary conditions.

For $(t, y, c) \in [0,1] \times \RR^2$, we set
\be
\Gamma (t, y, c) := p' \left[ (c-h(t)) (y^+)^{\frac{1}{p}} -f(t)  \right].
\notag
\ee
Before studying \eqref{ode}, let us first consider the following initial value problem which depends on a parameter $c \in \RR$:
\begin{equation}\label{y1}
y'(t) =  \Gamma (t, y, c), \quad  y(1) =0,
\end{equation}
which is referred to as the backward initial value problem. We notice that there exists a local solution, in the sense of Carath\'{e}odory, of \eqref{y1} if the so--called Carath\'{e}odory conditions are satisfied:
\begin{itemize}
\item[$(a)$] for almost every $t \in [0,1]$ fixed, $\Gamma(t, \cdot, \cdot)$ is continuous in $y$ and $c$,
\item[$(b)$] for every $y \in \RR$ and $c \in \RR$ fixed, $\Gamma(\cdot, y, c)$ is measurable in $t$, and
\item[$(c)$] there is a Lebesgue--integrable function $m(t)$ satisfying $|\Gamma(t, y, c)| \leq m(t)$ for all $t \in [0,1]$ and $y \in \RR$.
\end{itemize}
If $f, h \in L^1(0,1)$, then the Carath\'{e}odory conditions $(a)$ and $(b)$ will be clearly satisfied.  For any fixed $c \in \RR$, let $y_c=y_c(t)$ denote the solution of \eqref{y1} in the sense of Carath\'{e}odory. The global existence result of $y_c$ on $[0,1]$ was shown in \cite[Lemma 4.1]{DZ22} without the convection, i.e., $h(t) \equiv 0$. Similarly, we can have the following existence result.

\begin{lemma} \label{cara_existence} Let $f \in L^1(0,1)$, $h \in L^\infty(0,1)$, and $c \in \RR$ be fixed. Then the backward initial value problem \eqref{y1} has at least one global solution $y_c=y_c(t)$ on $[0,1]$ in the sense of Carath\'{e}odory.
\end{lemma}
\begin{proof}
First, following the proof of \cite[Lemma 4.1]{DZ22} by replacing $|c|$ with $|c|+ \|h\|_{L^{\infty}(0,1)}$, we can prove an a priori estimate for all solutions of \eqref{y1}. For $\s \in (0,1)$, integrating \eqref{y1} from $\sigma$ to $1$ gives
\be \label{y(sigma)}
y(\s)=-p'\int_{\s}^1 (c-h(\tau))(y^+(\tau))^{\frac{1}{p}}d\tau +p'\int_{\s}^1 f(\tau) d\tau.
\ee
For $t \in (0, 1)$, set
\be
\rho(t):=\max_{\s \in [t,1]} |y(\sigma)|.
\notag
\ee
It follows from \eqref{y(sigma)} that for $\s \in [t,1]$,
\be
|y(\s)|\leq p'(|c|+\|h\|_{L^{\infty}(0,1)}) \int_{\s}^1(y^+(\tau))^{\frac{1}{p}}d\tau +p'\|f\|_{L^1(0,1)},
\notag
\ee
from which we deduce that for $t \in (0,1)$,
\be
\begin{split}
\rho(t)
& \leq p'(|c|+\|h\|_{L^{\infty}(0,1)}) \max_{\s \in [t,1]} \int_{\s}^1(y^+(\tau))^{\frac{1}{p}}d\tau +p'\|f\|_{L^1(0,1)} \\
& \leq p'(|c|+\|h\|_{L^{\infty}(0,1)}) \int_{0}^1 \max_{\s \in [t,1]}  (y^+(\s))^{\frac{1}{p}}d\tau +p'\|f\|_{L^1(0,1)} \\
& \leq p'(|c|+\|h\|_{L^{\infty}(0,1)})( \rho(t) )^{\frac{1}{p}}+p'\|f\|_{L^1(0,1)}.
\notag
\end{split}
\ee
Since $\frac{1}{p}<1$, there is a constant $K>0$ such that $\rho(t) < K$ for all $t \in [0, 1]$. That is, all solutions of \eqref{y1} are a priori bounded by  $K$.
Set
\be
\tilde \Gamma(t, y, c):=
\begin{cases}
\Gamma(t, y, c), & \quad |y| < K, \\
p'\Big[(c-h(t))K^{\frac{1}{p}}-f(t)\Big], & \quad y \geq K, \\
-p'f(t), & \quad y \leq -K.
\end{cases}
\notag
\ee
It can be easily verified by the same procedure as above that the solutions of
\be \label{my1}
y'(t) =\tilde \Gamma(t, y, c), \quad y(1)=0,
\ee
are also bounded by this constant $K$ and the set of solutions of the modified problem \eqref{my1} coincides with the set of solutions of \eqref{y1}.
But $\tilde \Gamma$ satisfies Carath\'{e}odory conditions, in particular, there exists a function $m(t) \in L^1(0,1)$ satisfying $|\tilde \Gamma(t, y, c)| \leq m(t)$ for all $t \in [0,1]$ and $y \in \RR$. According to \cite[Theorem 10.XVIII]{W}, the backward initial value problem \eqref{my1} has a global solution in $[0,1]$ in the sense of Carath\'{e}odory. 
\end{proof}

\begin{remark} \label{R1} The same existence result holds for the forward initial value problem
\begin{equation}\label{y0}
y'(t) =  \Gamma (t, y, c), \quad  y(0) =0,
\end{equation}
by integrating \eqref{y0} from $0$ to $\s \in (0,1)$. However, throughout this section, we will concentrate on the backward initial value problem \eqref{y1}, rather than \eqref{y0},  in order to prove that $y_c>0$ in $(0,1)$ in Lemma \ref{positiveness}. Later we determine the region of $c$ for which \eqref{y1} has a positive solution $y_c$ that satisfies $y_c(0)=0$.
\end{remark}

Next, we prove that for any fixed $c \in \RR$, the solution $y_c$ of the backward initial value problem \eqref{y1} is in fact strictly positive in $(0,1)$. In this section the positivity of $y_c$  plays a key role in studying the uniqueness of solutions and comparison arguments of \eqref{y1}. This is rather different from the case $h(t) \equiv 0$ (\cite{DZ22}). Without the convection term, a necessary condition $c >0$ implies that a function $y \mapsto c(y^+)^{\frac{1}{p}}$, $y \in \RR$, is one--sided Lipschitz, and therefore it guaranteed the uniqueness of solutions of \eqref{y1} for all $c > 0$ (\cite[Lemma 4.3]{DZ22}). Furthermore, it led to the comparison argument that was one of the basic tools throughout \cite{DZ22}, in particular, when they proved the positivity of solutions of \eqref{y1} (\cite[Lemma 4.4 -- Theorem 4.7]{DZ22}). However, in our case, we lose the one--sided Lipschitz condition due to a possible sign change of $c-h(t)$ in \eqref{y1}. This difficulty will be surmounted by first proving, without such comparison results as in \cite{DZ22}, that $y_c >0$ in $(0,1)$ for any $c \in \RR$. In fact, we have realized that the positivity can be verified directly from \eqref{y1} regardless of the sign of $c-h(t)$.

\begin{lemma} \label{positiveness}
Let $f>0$ be lower semicontinuous in $(0,1)$, $f \in L^1(0,1)$, $h \in L^\infty(0,1)$, and $c \in \RR$. If $y_c=y_c(t)$ is a solution of the backward initial value problem \eqref{y1}, then $y_c(t) >0$ in $(0, 1)$.
\end{lemma}

\begin{proof} We argue by contradiction. Assume that there exists $t_0 \in (0,1) $ such that $y_c(t_0) \leq 0.$ First, we consider  the case $y_c(t_0) <0.$
Since $y_c$ is continuous in $(0,1),$ there exists an interval  $ t_0 \in (t_1, t_2) \subseteq (0,1)$
such that $y_c(t) <0$ for all $t \in (t_1,t_2)$ and $y_c(t_2) =0.$
Integrating \eqref{y1} from $t_1$ to $t_2$ yields that
\begin{equation*}
0 \leq  \int_{t_1}^{t_2} y'_c (\tau) d\tau = -p' \int_{t_1}^{t_2} f(\tau) d\tau  <0,
\end{equation*}
which is a contradiction.

\smallskip

Next, for the case $y_c (t_0) =0$ we may assume that $y_c (t)\geq0$ in $(t_0,1)$.
Otherwise, a contradiction is induced by the argument used in the first case.
Since $f$ is lower semicontinuous and $f>0$ in $(0,1)$, there is a positive number $\eta$ such that
\be\label{ineq_f}
f(t) \geq \eta>0 \quad \text{for all $t \in [t_0, 1-\varepsilon]$}
\ee
for given $\varepsilon>0$ sufficiently small. Again integrating \eqref{y1} from $t_0$ to $t \in (t_0, 1-\varepsilon]$ we have
\be
y_c(t)=p'\int_{t_0}^t (c-h(\tau))(y_c(\tau))^{\frac{1}{p}}d\tau -p' \int_{t_0}^t f(\tau) d\tau,
\notag
\ee
or equivalently
\be \label{int2}
\frac{y_c(t)}{t-t_0}=p'\Big[\frac{\int_{t_0}^t (c-h(\tau))(y_c(\tau))^{\frac{1}{p}}d\tau}{t-t_0} - \frac{\int_{t_0}^t f(\tau) d\tau}{t-t_0}\Big].
\ee
We observe that for all $t \in (t_0, 1-\varepsilon]$
\be
( c- h_M ) \int_{t_0}^{t}(y_c(\tau))^{\frac{1}{p}}d\tau  \leq \int_{t_0}^{t} (c-h(\tau)) (y_c(\tau))^{\frac{1}{p}}d\tau
 \leq (c - h_m) \int_{t_0}^{t} (y_c(\tau))^{\frac{1}{p}}d\tau.
\notag
\ee
Here, as stated in \eqref{maxminh}, $h_M$ and $h_m$ are supremum and infimum of $h$ in $[0,1]$. Since $y_c \in C[0,1]$ and $y_c(t_0)=0$,
\be
\lim_{t \rightarrow t_0^+} \frac{\int_{t_0}^t (y_c(\tau))^{\frac{1}{p}}d\tau}{t-t_0}=\lim_{t \rightarrow t_0^+} (y_c(t))^{\frac{1}{p}}= 0,
\notag
\ee
from which we deduce that
\be \label{lim1}
\lim_{t \rightarrow t_0^+} \frac{\int_{t_0}^t (c-h(\tau))(y_c(\tau))^{\frac{1}{p}}d\tau}{t-t_0}= 0.
\ee
Due to \eqref{ineq_f}, we immediately see that
\be \label{lim2}
\frac{\int_{t_0}^t f(\tau) d\tau}{t-t_0} \geq \eta>0.
\ee
Hence we conclude from \eqref{int2}--\eqref{lim2} that for $t$ close enough to $t_0$ with $t>t_0$,
\be
\frac{y_c(t)}{t-t_0}<0,
\notag
\ee
in contradiction to the fact that $y_c (t)\geq 0$ in $(t_0,1)$. Therefore, $y_c(t) >0$ on $(0,1)$.
\end{proof}

\begin{remark} \label{R2} As seen in the proof of Theorem \ref{positiveness}, the positivity of solutions holds only for the backward initial value problem \eqref{y1} because $f >0$ in $(0,1)$. For this reason, in order to solve the boundary value problem \eqref{ode} we start with the backward initial value problem \eqref{y1}.
\end{remark}

\smallskip

Thanks to the positivity of solutions of \eqref{y1}, we have the following uniqueness result.

\begin{lemma} \label{Uniqueness}
Let $f>0$ be lower semicontinuous in $(0,1)$, $f \in L^1(0,1)$, $h \in L^\infty(0,1)$, and $c \in \RR$. Then the backward initial value problem \eqref{y1} has a unique positive solution $y_c=y_c(t)$.
\end{lemma}
\begin{proof} Suppose that both $y_1=y_1(t)$ and $y_2=y_2(t)$ are solutions of \eqref{y1}. The positivity of solutions allows us to set
\be
z_1(t):=(y_1(t))^{\frac{1}{p'}}>0 \quad \text{and} \quad z_2 (t):=(y_2(t))^{\frac{1}{p'}}>0, \quad t \in (0,1).
\notag
\ee
Recalling \eqref{y1} we directly compute
\be
z_1'(t) =c-h(t)-\frac{f(t)}{(z_1(t))^{1/(p-1)}} \quad \text{and} \quad
z_2'(t)  =c-h(t)-\frac{f(t)}{(z_2(t))^{1/(p-1)}}, \quad a.e. ~t \in (0,1).
\notag
\ee
We now set
\be
\omega(t):=(z_1(t)-z_2(t))^2, \quad t \in [0,1].
\notag
\ee
Then $\omega(t) \geq 0$ in $[0,1]$, and $\omega(1)=0$. Moreover, since $f>0$ in $(0,1)$, for a.e. $t \in (0,1)$
\be
\begin{split}
\omega'(t)
& =2(z_1(t)-z_2(t))(z_1'(t)-z_2'(t)) \\
& =-2f(t)(z_1(t)-z_2(t))\Big(\frac{1}{(z_1(t))^{1/(p-1)}}-\frac{1}{(z_2(t))^{1/(p-1)}}\Big) \\
& \geq 0,
\end{split}
\notag
\ee
from which we deduce that $\omega(t) =0$ a.e. in $[0,1]$. Therefore, $y_1(t)=y_2(t)$ in $[0,1]$.
\end{proof}

\smallskip

From the above uniqueness result, we immediately have continuous dependence of solutions on the parameter $c$.

\begin{lemma} \label{Cdop}
Let $c_0 \in \RR$ be fixed. As $c \rightarrow c_0$, the solutions $y_c=y_c(t)$ of \eqref{y1} converge to $y_{c_0}$ uniformly in $[0,1]$ $($i.e., in the topology of $C[0,1]$$)$.
\end{lemma}
\begin{proof} Since $\Gamma (t, y, c) $ is continuous in $c$, the proof follows from the previous uniqueness result and \cite[Theorem 4.1 and 4.2 in Chapter 2]{CL}.
\end{proof}

\smallskip

Following \cite{DZ22, W}, we introduce the notion of the defect $P_c \varphi$ of a function $\varphi=\varphi(t)$ with respect to the differential equation $y'(t)=\Gamma(t, y, c)$:
\be \label{defect}
P_c \varphi:=\varphi'(t)-\Gamma(t, \varphi(t),c).
\ee
In particular, $P_c y_c=0$ a.e. in $[0,1]$ if  $y_c$ satisfies $y'(t)=\Gamma(t, y, c)$. Using this notion we deduce the following comparison results between positive functions.

\begin{lemma} \label{lecom} Let $f>0$ be lower semicontinuous in $(0,1)$, $f \in L^1(0,1)$, $h \in L^{\infty}(0,1)$, and $c \in \RR$. Assume that $\varphi=\varphi(t)>0$ in $(0,1)$. Then the following two alternatives hold:

\smallskip

\noindent $(a)$ if
\be \label{super}
\varphi(1) \geq 0 \quad  \text{and} \quad P_c \varphi \leq 0 \quad \text{a.e. in $[0,1]$},
\ee
then $y_c \leq \varphi$ in $[0,1]$;

\smallskip

\noindent $(b)$ if
\be \label{sub}
\varphi(1) \leq 0 \quad \text{and} \quad  P_c \varphi \geq 0 \quad \text{a.e. in $[0,1]$},
\ee
then $y_c \geq \varphi$ in $[0,1]$.
\end{lemma}
\begin{proof} We only prove $(a)$. The proof of $(b)$ follows similarly. We first notice from Lemma \ref{positiveness} that $y_c=y_c(t) >0$ in $(0,1)$. Similarly as in Lemma \ref{Uniqueness}, we set
\be
z_1(t):=(\varphi(t))^{\frac{1}{p'}}, \quad  z_2(t):=(y_c(t))^{\frac{1}{p'}}, \quad \text{and} \quad  \o(t):=z_1(t)-z_2(t)
\notag
\ee
for all $t \in (0,1)$. In order to verify inequality $\o(t)\geq 0$ in $(0,1)$, we assume by contradiction that there exists $t_0 \in (0,1)$ such that $\o(t_0) < 0$. Let $t_1 \in (t_0, 1]$ be such that $\o(t) \leq 0$ on $(t_0, t_1]$. It follows from \eqref{super} that for a.e. $t \in (0,1)$,
\be
z_1'(t) \leq c-h(t)-\frac{f(t)}{(z_1(t))^{1/(p-1)}} \quad \text{and} \quad
z_2'(t)  =c-h(t)-\frac{f(t)}{(z_2(t))^{1/(p-1)}},
\notag
\ee
from which we deduce that
\be
\o'(t)=z_1'(t)-z_2'(t) \leq -f(t)\Big(\frac{1}{(z_1(t))^{1/(p-1)}}-\frac{1}{(z_2(t))^{1/(p-1)}} \Big) \leq 0 \quad \text{a.e. in $(t_0, t_1]$},
\notag
\ee
and consequently $\o(t_1) \leq \o(t_0)<0$. Applying the same argument repeatedly if necessary, we conclude that $\o(1)<0$, in contradiction to the fact that $\o(1) \geq 0$.
\end{proof}

\begin{remark} In Lemma \ref{lecom}$(a)$, the assumption $\varphi>0$ in $(0,1)$ can be removed. It can be proved from \eqref{super} by the similar argument to Lemma \ref{positiveness}. Unlike $(a)$, the positivity assumption of $\varphi$ in $(b)$ cannot be removed. It seems to give more restriction as compared to \cite{DZ22}, but all functions $\varphi$ we will construct for the comparison are positive.
\end{remark}

\smallskip

From Lemma \ref{lecom}$(a)$, we see that the function $c \mapsto y_c(t): \RR \rightarrow \RR$ is monotone decreasing.

\begin{corollary} \label{cocom} Let $c_1, c_2 \in \RR$, and $y_{c_i}$ be a solution of \eqref{y1} for $i=1, 2$. If $c_2 \geq c_1$, then $y_{c_1}(t) \geq y_{c_2}(t)$ for all $t \in [0,1]$.
\end{corollary}
\begin{proof} We first notice from Lemma \ref{positiveness} that $y_{c_i}>0$ in $(0,1)$ for $i=1, 2$.   Recalling \eqref{defect}, it follows from $P_{c_1}y_{c_1}=0$ that
\be
\begin{split}
P_{c_2}y_{c_1}
& = y_{c_1}'-\Gamma(t, y_{c_1},c_2) \\
& =y_{c_1}'-\Gamma(t, y_{c_1},c_1)+\Gamma(t, y_{c_1},c_1)-\Gamma(t, y_{c_1},c_2) \\
& = \Gamma(t, y_{c_1},c_1)-\Gamma(t, y_{c_1},c_2)  \\
& = p' \left[ (c_1-h(t)) (y_{c_1})^{\frac{1}{p}} -f(t)  \right]-p' \left[ (c_2-h(t)) (y_{c_1})^{\frac{1}{p}} -f(t)  \right] \\
& = p'(c_1-c_2)(y_{c_1})^{\frac{1}{p}} \leq 0.
\end{split}
\notag
\ee
Therefore, by Lemma \ref{lecom}(a),   $y_{c_1}(t) \geq y_{c_2}(t)$ for all $t \in [0,1]$.
\end{proof}

\smallskip

We are now ready to prove the existence and nonexistence results for the first order boundary value problem \eqref{ode}. By modifying the proofs for the case $h \equiv 0$ (\cite[Theorem 4.7 and 4.8]{DZ22}), we determine the threshold wave speed $c^*$ for the existence and nonexistence of solutions of \eqref{ode}.  Before stating we recall notations $h_M$ and $h_m$ from \eqref{maxminh}.

\begin{theorem} [Existence] \label{exthm}
Let $f>0$ be lower semicontinuous in $(0,1)$, and $h$ satisfies \textbf{(H3)}. Set
\be \label{mu}
0< \mu := \sup_{t\in (0,1)} \frac{f(t)}{t^{p'-1}} < + \infty.
\ee
Then, in the case $h(0)>H(1)$, there exists a number
\be\label{c*}
c^* \in [ h(0), ~ h_M +(p')^{\frac{1}{p'}} p^{\frac{1}{p} }\mu^{\frac{1}{p'}}]
\ee
such that the boundary value problem \eqref{ode} has a unique positive solution if and only if $c\geq c^*$. In the case $h(0) \leq H(1)$, the same result holds for
\be \label{H(1)}
c^* \in (H(1), ~h_M+(p')^{\frac{1}{p'}} p^{\frac{1}{p} }\mu^{\frac{1}{p'}}  ].
\ee
\end{theorem}
\begin{proof} We notice that the condition \eqref{mu} implies that $f \in L^1(0,1)$. First, we claim that $y_c(0) =0 $ for $c$ sufficiently large, that is, a solution $y_c$ of \eqref{y1} turns out to be a positive solution of \eqref{ode} for $c$ sufficiently large. For $c > h_M$, set
\be
\phi_c (s):= (c- h_M) s^{ \frac{1}{p} } -s, \quad s \in (0, (c-h_M)^{p'}).
\notag
\ee
Then $\phi_c(s) >0$ for all $s \in (0,  (c- h_M)^{p'})$, and $\phi_c(0) = \phi_c( (c-h_M)^{p'})=0$. Also, we readily check that $\phi_c$ attains the maximum value
$$M_c := \left( \frac{c-h_M}{p} \right)^{p'}(p-1)~  \mbox{at the point} ~ k:= \left(\frac{c- h_M}{p} \right)^{p'} \in (0,  (c- h_M)^{p'}).$$
A short calculation shows that $c- h_M \geq (p')^{\frac{1}{p'}} p^{\frac{1}{p}} \mu^{\frac{1}{p'}}$ if and only if $M_c \geq \mu$, from which we deduce that if $c \geq h_M+ (p')^{\frac{1}{p'}} p^{\frac{1}{p}} \mu^{\frac{1}{p'}}$ then $\phi_c (k) \geq \mu$, that is,
\begin{equation} \label{bd}
(c- h_M) k^{\frac{1}{p}} -\mu \geq k.
\end{equation}
Let $\varphi(t) := kt^{p'}$.  Then $\varphi(1) >0$, and recalling \eqref{defect},
\be
\begin{split}
P_c \varphi
& = kp' t^{p'-1} -p' \Big[(c-h(t))(\varphi(t))^{\frac{1}{p}} -f(t)\Big]\\
& \leq \Big[ (c-h_M) k^{\frac{1}{p}} -\mu \Big]p' t^{p'-1} -p' \Big[(c-h_M)(\varphi(t))^{\frac{1}{p}} -f(t)\Big] \\
& \leq \Big[ (c-h_M) k^{\frac{1}{p}} -\mu \Big]p' t^{p'-1} -p' \Big[(c-h_M)(\varphi(t))^{\frac{1}{p}} - \mu t^{p'-1}  \Big] \\
& =0.
\end{split}
\notag
\ee
Here, we have used \eqref{mu} and \eqref{bd}. Then it follows by Lemma \ref{lecom}(a) that
\be
0 \leq y_c(t ) \leq \varphi(t), \quad   t \in [0,1].
\notag
\ee
Since $\varphi(0)=0$, we conclude that $y_c(0)=0$. Therefore, for any $c \geq   h_M+(p')^{\frac{1}{p'}} p^{\frac{1}{p}} \mu^{\frac{1}{p'}}$, a unique solution $y_c=y_c(t)$ of \eqref{y1} is also a positive solution of \eqref{ode}.

Now, recalling \eqref{ncs}, set
\be
c^* :=  \inf \{ c \geq h(0) ~ \text{and} ~ c>H(1) : ~ \mbox{\eqref{ode} has a unique positive solution}\}.
\notag
\ee
Then we immediately see that $\max\{h(0), H(1)\} \leq c^* \leq  h_M+(p')^{\frac{1}{p'}} p^{\frac{1}{p}} \mu^{\frac{1}{p'}}$. Let $c_n \geq c_*$ be a sequence that converges to $c_*$ and $y_{c_n}=y_{c_n}(t)$ be a positive solution of \eqref{ode}. It follows from Lemma \ref{Cdop} that $y_{c_n}$ converges to $y_{c^*}$ uniformly in $[0,1]$, and hence $y_{c^*}$ is also a positive solution of \eqref{ode}. Consequently, \eqref{ode} has a unique positive solution if and only if $c \geq c^*$. However, from \eqref{c>H(1)} we get $c \neq H(1)$, so that we arrive at \eqref{H(1)} in the case $h(0) \leq H(1)$.
\end{proof}

\begin{remark} As established in \cite{DZ22}, if $h(t)=0$ in $[0,1]$, then the threshold satisfies
\be
c^* \in (0, ~(p')^{\frac{1}{p'}} p^{\frac{1}{p} }\mu^{\frac{1}{p'}}]
\notag
\ee
which is the special case of \eqref{H(1)}.
\end{remark}

\smallskip

\noindent
\textbf{Proof of Theorem \ref{Existence_U}.} Under the assumptions \textbf{(H1)}--\textbf{(H3)} the proof follows immediately from Proposition \ref{prop} and Theorem \ref{exthm}. In particular, the function $f(t)=(d(t))^{\frac{1}{p-1}}g(t)$ satisfies all the hypotheses of Theorem \ref{exthm} and the properties of $U$ are shown in derivation of Proposition \ref{prop}. As defined in \eqref{z(U)},  $U(0)=\frac{1}{2}$ is a normalized condition. Indeed, by translation invariance of \eqref{eqU}, if $U=U(z)$ is a solution of \eqref{eqU} and \eqref{bc}, then the function $z \mapsto U(\eta+z)$ is also a solution for any fixed $\eta \in \RR$.

\smallskip

\begin{theorem}[Nonexistence] \label{Nonexistence} Let $f>0$ be lower semicontinuous in $(0,1)$, $f \in L^1(0,1)$, and $h$ satisfies \textbf{(H3)}. Set
\be \label{nu}
0< \nu:=\liminf_{t \rightarrow 0^+}\frac{f(t)}{t^{p'-1}}.
\ee
If
\be \label{condition_nonexist}
h(0) < c < h(0)+(p')^{\frac{1}{p'}}p^{\frac{1}{p}}\nu^{\frac{1}{p'}},
\ee
then the boundary value problem \eqref{ode} has no positive solution. In particular, if $\nu =+\infty$, then the problem \eqref{ode} has no positive solution for any $c >h(0)$.
\end{theorem}
\begin{proof} We argue by contradiction. Fix any $c \in \mathbb{R}$ that satisfies \eqref{condition_nonexist}. Assume that the boundary value problem \eqref{ode} has a positive solution $y_c=y_c(t)$ in $(0,1)$.  Since $0 < c -h(0)< (p')^{\frac{1}{p'}}p^{\frac{1}{p}}\nu^{\frac{1}{p'}}$,  there exists a constant $L>0$ such that 
\be
0 < c -h(0)< L<(p')^{\frac{1}{p'}}p^{\frac{1}{p}}\nu^{\frac{1}{p'}}. 
\notag
\ee
By the continuity of $h$ at $0$, there exist $t_0 \in (0,1]$ such that  
\be \label{range of c-h(t)}
0<c-h(t)<L<(p')^{\frac{1}{p'}}p^{\frac{1}{p}}\nu^{\frac{1}{p'}} \quad \text{for all  $t \in [0, t_0]$}. 
\ee
Since $y_c$ is also a solution of \eqref{y1}, it follows from the uniqueness result that the restricted function $\tilde y_c:=y_c|_{[0, t_0]}$ is a unique solution of the following backward initial value problem
\be
\begin{cases}
\tilde y'(t) = p' \left[ (c-h(t)) (\tilde y^{+} (t))^{\frac{1}{p}} -f(t)  \right] \quad \text{a.e. in $(0,t_0)$},\\
\tilde y(t_0)=y_c(t_0).
\end{cases}
\notag
\ee
We now define a operator $T: C[0,t_0] \rightarrow C[0, t_0]$ as
\be
T(u)(t):=p'\int_0^t\Big[(c-h(\tau))(u^+(\tau))^\frac{1}{p}-f(\tau)\Big] d\tau, \quad u \in C[0,t_0].
\notag
\ee
Then $\tilde y_c$ is definitely a fixed point of the operator $T$ in $C[0,t_0]$. Owing to the fact that $c-h(t) > 0$ on $[0,t_0)$,  we conclude that for any $u_1\leq u_2 $ on $[0,t_0]$,
\be
T(u_1)(t)-T(u_2)(t)=p'\int_0^t\Big[(c-h(\tau))\Big((u_1^+(\tau))^\frac{1}{p}-(u_2^+(\tau))^\frac{1}{p}\Big) \Big] d\tau \leq 0,
\notag
\ee
that is, $T$ is monotone increasing on $C[0,t_0]$. Set
\be
\tilde y_0(t):=L^{p'}t^{p'}, \quad t \in [0,t_0]. 
\notag
\ee
Due to \eqref{range of c-h(t)}, we obtain that for all $t \in [0, t_0]$, 
\be
\begin{split}
T(\tilde y_0)(t)
& = p'\int_0^t\Big[(c-h(\tau))  L^{\frac{p'}{p}}\tau^{\frac{p'}{p}} \Big] d\tau-p'\int_0^t f(\tau) d\tau \\
& \leq p'L^{1+\frac{p'}{p}}\int_0^t   \tau^{\frac{p'}{p}} d\tau   \\
& = \tilde y_0(t). 
\end{split}
\notag
\ee
Then, by the monotonicity of $T$ on $C[0,t_0]$, a sequence defined as
\be \label{seq_y_n}
\tilde y_{n+1}=T(\tilde y_n), \quad n=0, 1, 2, \cdots,
\ee
satisfies that
\be
\tilde y_0(t) \geq \tilde y_1(t) \geq \cdots \geq \tilde y_n(t) \geq \cdots,  \quad t \in (0,t_0),
\notag
\ee
and moreover it is bounded below; namely
\be
\tilde y_{n}=T(\tilde y_{n-1}) \geq -p'\int_0^t f(\tau) d\tau, \quad n \in \mathbb{N}.
\notag
\ee
According to \cite[Theorem 6.3.16]{DM}, the sequence \eqref{seq_y_n} converges to the greatest fixed point of $T$. From the fact that $\tilde y_c$ is a fixed point of $T$, we deduce that
\be \label{mtildey}
\tilde y_0(t) \geq \tilde y_1(t) \geq \cdots \geq \tilde y_n(t) \geq \cdots \geq \tilde y_c(t) >0, \quad t \in (0,t_0).
\ee

\smallskip

In order to obtain a contradiction we notice from \eqref{nu} and \eqref{range of c-h(t)} that there exists $\delta \in (0,t_0]$ and $\tilde \nu \in \big(\frac{1}{p'p^{p'-1}}, 1\Big)$ such that
\be \label{ineq_f_nu}
f(t) \geq \tilde \nu L^{p'}t^{p'-1}, \quad t \in (0, \delta).
\ee
Recalling \eqref{seq_y_n} and using \eqref{ineq_f_nu}, a direct computation gives that
\be
\tilde y_1(t) \leq L^{p'}t^{p'}(1-\tilde \nu),  \quad t \in (0, \delta),
\notag
\ee
and
\be
\tilde y_2(t) \leq L^{p'}t^{p'}\Big[(1-\tilde \nu)^{\frac{1}{p}}-\tilde \nu \Big], \quad t \in (0, \delta).
\notag
\ee
In general, for each $n \in \mathbb{N}$, we obtain
\be \label{mtildey2}
\tilde y_n(t) \leq a_n L^{p'}t^{p'}, \quad  \quad t \in (0, \delta),
\ee
where $a_n$ is a sequence defined as
\be \label{seq_a_n}
a_0=1, \quad a_n=(a_{n-1})^{\frac{1}{p}}-\tilde \nu.
\ee
It follows from \eqref{mtildey} and \eqref{mtildey2} that the sequence $\{ a_n \}$ is decreasing and bounded below by $\tilde y_c$, and hence it converges to $a_{\infty} \in (0,1)$ that satisfies $\tilde \nu=a_{\infty}^{\frac{1}{p}}-a_{\infty}=a_{\infty}^{\frac{1}{p}}(1-a_{\infty})^{\frac{1}{p'}}$. A simple calculation shows that
\be
\tilde \nu \leq \max_{x \in (0,1)} x^{\frac{1}{p}}(1-x)^{\frac{1}{p'}}=\frac{1}{p'p^{p'-1}},
\notag
\ee
a contradiction to the fact that $\tilde \nu \in \big(\frac{1}{p'p^{p'-1}}, 1\big)$.
\end{proof}


\begin{remark} If $H(1) < h(0)+(p')^{\frac{1}{p'}}p^{\frac{1}{p}}\nu^{\frac{1}{p'}}$, then we conclude from Theorem \ref{exthm} and Theorem \ref{Nonexistence} that the threshold speed $c^*$ satisfies
\be \label{range of c^*}
h(0)+(p')^{\frac{1}{p'}}p^{\frac{1}{p}}\nu^{\frac{1}{p'}} \leq c^* \leq  h_M+(p')^{\frac{1}{p'}} p^{\frac{1}{p} }\mu^{\frac{1}{p'}}.
\ee
Otherwise, i.e., if $h(0)+(p')^{\frac{1}{p'}}p^{\frac{1}{p}}\nu^{\frac{1}{p'}} \leq H(1)$, then the region of $c^*$ is determined only by \eqref{H(1)}. Figure \ref{figure_c*} depicts the region of $c^*$.
\end{remark}

\begin{remark} The situation described by  \eqref{range of c^*} with $p=p'=2$ occurs, for instance, if we take
\be
g(u)=
\begin{cases}
u^2, & 0 \leq u \leq \frac{1}{2}, \\
\frac{1}{2}(1-u), & \frac{1}{2} \leq u \leq 1, 
\end{cases}
\quad d(u)=
\begin{cases}
\frac{1}{6u}, & 0 \leq u \leq \frac{1}{3}, \\
4u(1-u), & \frac{1}{3} < u \leq 1, 
\end{cases}
\notag
\ee
and
\be
h(u)=
\begin{cases}
u+\frac{1}{2}, & 0 \leq u \leq \frac{1}{2}, \\
\frac{3}{4}-\frac{u}{2}, & \frac{1}{2} < u \leq 1, 
\end{cases}
\notag
\ee
in which case $\mu=\frac{1}{2}$, $\nu=\frac{1}{6}$, $h(0)=\frac{1}{2}$, $h_M=1$, and $H(1)=\frac{9}{16}$. Since $\frac{9}{16}=H(1) < h(0)+(p')^{\frac{1}{p'}}p^{\frac{1}{p}}\nu^{\frac{1}{p'}}=\frac{3+2\sqrt{6}}{6}$, the threshold speed $c^*$ satisfies $\frac{3+2\sqrt{6}}{6} \leq c^* \leq 1+\sqrt{2} $. 
\end{remark}

\smallskip

\noindent
As in the case of the existence result, the \textbf{proof of Theorem \ref{Nonexistence_U}} follows from Proposition \ref{prop} and Theorem \ref{Nonexistence}.

\smallskip


\begin{figure}
  \centering
 \includegraphics[scale=0.87]{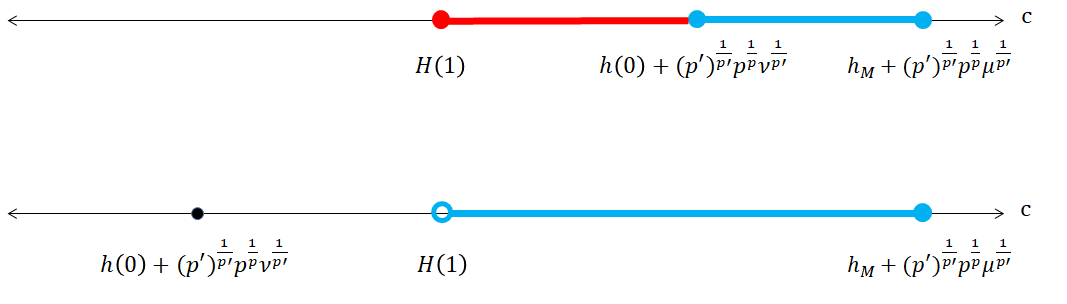}
\caption{The region of the threshold speed $c^*$, as discussed in Theorem \ref{exthm}, is shown as a blue line. In the case that $H(1) < h(0)+(p')^{\frac{1}{p'}}p^{\frac{1}{p}}\nu^{\frac{1}{p'}}$ (top figure), the problem \eqref{ode} has no positive solution for any $c$ between $H(1)$ and $h(0)+(p')^{\frac{1}{p'}}p^{\frac{1}{p}}\nu^{\frac{1}{p'}}$ (red line), so that the region of $c^*$ described in Theorem \ref{exthm} is reduced to \eqref{range of c^*}. Otherwise, the nonexistence result does not effect the region of $c^*$ (bottom figure).}
 \label{figure_c*}
\end{figure}

\smallskip

We now consider the special case when the threshold wave speed $c^*$ becomes $h(0)$.

\begin{theorem}[Threshold $c^*=h(0)$] \label{threshold2} Suppose that all the hypotheses in Theorem \ref{exthm} hold.  If $h(0) > h(t)$ for all $t \in (0, 1)$ and 
\be \label{threshold_h(0)}
f(t) \leq \frac{1}{p'p^{p'/p}} (h(0)-h(t)) \Big[\int_0^t (h(0)-h(\tau)) d\tau \Big]^{\frac{p'}{p}}, \quad  t \in [0,1],
\ee
then $c^*=h(0)$.
\end{theorem}
\begin{proof} We first notice that $H(1)<h(0)$. It is enough to prove that there is a positive solution $y_c$ of \eqref{ode} when $c=h(0)$. By Lemma \ref{cara_existence}--Lemma \ref{Uniqueness}, for $c=h(0)$ there exists a unique positive solution $y_c$ of the backward initial value problem \eqref{y1}.

Define a function
\be \label{K}
K(t)=\Big[ \frac{1}{p}\int_0^t (h(0)-h(\tau)) d\tau \Big]^{p'}, \quad \quad  t \in [0,1].
\ee
Then $K(t) > 0$ in $(0,1)$ and $K(1) \geq 0$. The inequality \eqref{threshold_h(0)} gives that
\be
\begin{split}
K'(t)
& =\frac{p'}{p} (h(0)-h(t)) \Big[\frac{1}{p}\int_0^t (h(0)-h(\tau)) d\tau \Big]^{\frac{p'}{p}} \\
& = p'(h(0)-h(t))\Big[\frac{1}{p}\int_0^t (h(0)-h(\tau)) d\tau \Big]^{\frac{p'}{p}}- (h(0)-h(t)) \Big[\frac{1}{p}\int_0^t (h(0)-h(\tau)) d\tau \Big]^{\frac{p'}{p}} \\
& = p'(h(0)-h(t))\Big[\frac{1}{p}\int_0^t (h(0)-h(\tau)) d\tau \Big]^{\frac{p'}{p}}-\frac{p'}{p'p^{p'/p}} (h(0)-h(t)) \Big[\int_0^t (h(0)-h(\tau)) d\tau \Big]^{\frac{p'}{p}} \\
& \leq  p'\Big[(h(0)-h(t))(K(t))^{\frac{1}{p}}-f(t)\Big]  \quad \text{a.e. in $[0,1]$},
\notag
\end{split}
\ee
that is, $P_{h(0)}K \leq 0$ a.e. in $[0,1]$. It follows from Lemma \ref{lecom} that $0 \leq y_{h(0)}(t) \leq K(t)$ on $[0,1]$. Since $K(0)=0$, we conclude that $y_{h(0)}(0)=0$, that is, $y_{h(0)}$ is also a unique positive solution of \eqref{ode}.
\end{proof}

\begin{remark} If the assumptions of Theorem \ref{threshold2} hold, then  $\nu$ should be $0$. If not, it follows from Theorem \ref{Nonexistence} that the boundary value problem \eqref{ode} has no positive solution for all $c \in (h(0), h(0) +(p')^{\frac{1}{p'}}p^{\frac{1}{p}}\nu^{\frac{1}{p'}})$, in contradiction to the conclusion $c^*=h(0)$ of Theorem \ref{threshold2}. Indeed, by direct calculation,
\be
\frac{f(t)}{t^{p'-1}}  \leq \frac{1}{p'p^{p'/p}} (h(0)-h(t)) \Big[\frac{1}{t}\int_0^t (h(0)-h(\tau)) d\tau \Big]^{\frac{p'}{p}} \leq \frac{1}{p'p^{p'/p}} (h(0)-h(t)) (2|h(0)|)^{\frac{p'}{p}}. 
\ee
Since $h$ is continuous at $0$, letting $t \rightarrow 0^+$ yields $\nu=0$. 
\end{remark}

\begin{remark}If $p=p'=2$, then the inequality \eqref{threshold_h(0)} becomes
\be
f(t) \leq \frac{1}{4} (h(0)-h(t)) \int_0^t (h(0)-h(\tau)) d\tau, \quad  t \in [0,1],
\notag
\ee
which is the same as the assumption for $c^*=h(0)$ established in \cite[Theorem 1.2]{MM05}.
\end{remark}

\section{Asymptotic analysis of the traveling wave} \label{asymptotic analysis}

As stated in Theorem \ref{Existence_U}, the traveling wave profile $U$ is strictly decreasing on $(z_0, z_1) \subseteq \RR$. In this section we provide a short observation of the asymptotic behaviors near $U=1$ (i.e., $z_0=-\infty$ or $z_0 > -\infty$) and $U=0$ (i.e., $z_1=+\infty$ or $z_1 < +\infty$). The convection term does not conclusively affect the asymptotic analyses, so that all the proofs in this section are omitted, referring the interested reader to \cite[Section 6]{DZ22} for details. Thanks to Lemma \ref{lecom}, their proofs can be easily extended to our case. We only mention where the convection term affects their proofs.

In what follows,  we assume that all the assumptions in Theorem \ref{Existence_U} hold for the existence of the profile $U$. Following \cite{DZ22}, we only consider the power--type functions for the behaviors of $g=g(t)$ and $d=d(t)$ near $t=0$ and $t=1$. For notational brevity, we write
\begin{center}
$g_1(t) \sim g_2(t)$ as $t \rightarrow t_0$ if and only if $\ds \lim_{t \rightarrow t_0} \frac{g_1(t)}{g_2(t)} \in (0, +\infty)$.
\end{center}

\subsection{Asymptotics near $1$}

In this subsection we assume that the functions $g$ and $d$ satisfy
\be
g(t) \sim (1-t)^\gamma \quad \text{and} \quad d(t)\sim (1-t)^\delta
\notag
\ee
as $t \rightarrow 1^-$ for some $\gamma>0$ and $\delta \in \RR$. Then we immediately see from \eqref{mu_U} that
\be
0 \leq \gamma+\frac{\delta}{p-1}.
\notag
\ee

\smallskip

We first state the asymptotics of the profile $U$ near $1$.

\begin{theorem} Let $c \geq c^*$ be fixed. Suppose that $\gamma>0$, $\delta \in \RR$,  and
\be
0 \leq \gamma+\frac{\delta}{p-1} \leq \frac{1}{p-1}.
\notag
\ee
If $\ds \frac{\gamma-\delta+1}{p}<1$, then $z_0>-\infty$ and $U'(z_0+)=0$; otherwise $z_0=-\infty$.
\end{theorem}
\begin{proof} The proof is shown in \cite[Theorem 6.1 and Remark 6.4]{DZ22}.
\end{proof}

\begin{theorem} Suppose that $\gamma>0$, $\delta \in \RR$,  and
\be
\gamma+\frac{\delta}{p-1}>\frac{1}{p-1}.
\notag
\ee
Then the following two cases hold. \\
(i) If $\gamma \geq 1$, $c \geq c^*$ and there exists $\theta \in (0,1) $ such that $c>h(t)$ a.e. in $[1-\theta, 1]$, then $z_0=-\infty$.\\
(ii) If $\gamma<1$ and $c\geq c^*$, then $z_0>-\infty$; furthermore, under the additional assumption
\begin{center}
$\delta \leq 0$ \quad or \quad $c>h(t)$ a.e. in $[1-\theta,1]$ for some $\theta \in (0,1)$,
\end{center}
we conclude that $U'(z_0+)=0$, i.e., $U$ is a $C^1$-function in a neighborhood of $z_0 \in \RR$. 
\end{theorem}
\begin{proof} For the case $\gamma \geq 1$ the proof goes similarly as in \cite[Theorem 6.2]{DZ22} under the assumption $c-h(t) > 0$ near $t=1$. Indeed, the assumption allows us to construct $y_{\bar \kappa}$ introduced in \cite[Theorem 6.2]{DZ22}. The proof for the case $\gamma<1$ is essentially the same as in \cite[Theorem 6.2]{DZ22} to obtain that $z_0>-\infty$. However, in this case, if $\delta \leq 0$ then $liminf_{t \rightarrow 1^-} d(t)>0$. It follows immediately  from Theorem \ref{Existence_U} (iii) that $U'(z_0+)=0$. Even if $\delta>0$, assuming $c-h(t) > 0$ near $t=1$, we can still conclude $U'(z_0+)=0$ from the reasoning in \cite[Theorem 6.2 and Remark 6.4]{DZ22}. 
\end{proof}

\subsection{Asymptotics near $0$}

Next, we assume that the functions $g$ and $d$ satisfy
\be
g(t) \sim t^\a \quad  \text{and} \quad d(t)\sim t^\beta
\notag
\ee
as $t \rightarrow 0^+$ for some $\a>0$ and $\beta \in \RR$. For the existence of the profile $U$, $\nu < +\infty$ in \eqref{nu_U}, from which we get that
\be \label{alphabeta}
\a+\frac{\b}{p-1} \geq \frac{1}{p-1}.
\ee
By setting
\be \label{f1}
f_1:=\ds \sup_{t \in (0,1)} \frac{(d(t))^{\frac{1}{p-1}}g(t)}{t^{\a+\frac{\beta}{p-1}}},
\ee
we see from \eqref{mu_U} and \eqref{alphabeta} that $\mu \leq f_1 <+\infty$.

\smallskip

We now state the asymptotics of the profile $U$ near $0$ as follows.

\begin{theorem} Suppose that $\a>0$, $\beta \in \RR$, and \eqref{alphabeta} holds. Let $f_1$ be as in \eqref{f1}. Then the following two cases hold.\\
(i) If $\a\geq 1$ and $c \geq h_M+ (p')^{\frac{1}{p'}}p^{\frac{1}{p}}f_1^{\frac{1}{p'}}$, then $z_1=+\infty$. \\
(ii) If $0<\a<1$ and $c \geq c^*$, then $z_1<+\infty$; moreover, under the additional assumption 
\begin{center}
$\beta \leq 0$ \quad or \quad $c \geq h_M+ (p')^{\frac{1}{p'}}p^{\frac{1}{p}}f_1^{\frac{1}{p'}}$,
\end{center}
we have that $U'(z_1-)=0$, i.e., $U$ is a $C^1$-function in a neighborhood of $z_1 \in \RR$. 
\end{theorem}
\begin{proof} The proof can be obtained similarly as in \cite[Theorem 6.5]{DZ22} by replacing $c$ with $c-h_M$ for the first case and by replacing $\ds \Big(\frac{f_2}{c} \Big)^p$ with $\ds \Big(\frac{f_2}{c-h_m} \Big)^p$ for the second case, where $f_2$ is given in \cite[Theorem 6.5]{DZ22}. 
The statement concerning $U'(z_1-)$ follows from the fact $liminf_{t \rightarrow 0^+} d(t)>0$ for $\beta \leq 0$ and the same reasoning as in \cite[Theorem 6.5 and Remark 6.7]{DZ22}.
\end{proof}

\end{document}